\documentclass[12pt]{article}
\usepackage[utf8]{inputenc}
\usepackage{amsmath}
\usepackage{amssymb}
\usepackage{mathrsfs}
\usepackage{amsthm}
\usepackage{cite}
\usepackage{hyperref}
\usepackage{enumerate}
\usepackage[left=0.7in,right=0.7in,top=0.7in,bottom=0.7in]{geometry}
\usepackage{titlesec}
\titleformat*{\section}{\large\bfseries}
\newtheorem{theorem}{Theorem}[section]
\newtheorem{lemma}[theorem]{Lemma}

\newtheorem{definition}[theorem]{Definition}
\newtheorem{example}[theorem]{Example}
\newtheorem{proposition}[theorem]{Proposition}
\newtheorem{remark}[theorem]{Remark}
\numberwithin{equation}{section}
\title{On coupled best proximity points and Ulam-Hyers stability}
\author{\large Anuradha Gupta and Manu Rohilla$^*$}
\date{}
\begin{document}
\maketitle
\begin{abstract}
For two nonempty, closed, bounded and convex subsets $A$ and $B$ of a uniformly convex Banach space $X$ consider a mapping $T:(A \times B) \cup (B \times A) \rightarrow A \cup B$ satisfying $T(A,B) \subset B$ and $T(B, A) \subset A$. In this paper the existence of a coupled best proximity point is established when $T$ is considered to be a p-cyclic contraction  mapping and a p-cyclic nonexpansive mapping. The Ulam-Hyers stability of the best proximity point problem is also studied. 


\end{abstract}   
\section{Introduction and Preliminaries}
Let $A$ and $B$ be two nonempty subsets of a Banach space $X$. Define
\begin{align*}
dist(A,B)&:= \inf \{ \Vert a- b\Vert : a \in A \thinspace \thinspace \mbox{and} \thinspace \thinspace b \in B \},\\
A_0&:= \{ a \in A: \Vert a-b \Vert =dist(A,B) \thinspace \thinspace \mbox{for some } \thinspace \thinspace b \in B \},\\
B_0&:= \{ b \in B: \Vert a-b \Vert=dist(A,B) \thinspace \thinspace \mbox{for some} \thinspace \thinspace a \in A \}.
\end{align*}
 Let $T:(A \times B ) \cup (B \times A) \rightarrow A \cup B$ be a mapping. Then $(x,y) \in A \times B$ is called a coupled best proximity pair of $T$ if it satisfies
\begin{align*}
\Vert x-T(x,y) \Vert&=dist(A,B),\\
\Vert y-T(y,x) \Vert&=dist(A,B).
\end{align*}

Interestingly, it is a generalization of the concept of coupled fixed point \cite{12} under the assumption $A \cap B = \phi$.  If we take $A \cap B \neq \phi$, then the notion of coupled best proximity point reduces to that of coupled fixed point. 

Given a Banach space $(X, \Vert. \Vert)$, define a norm on $X \times X$ by $\Vert (x,y) \Vert= \max \{\Vert x \Vert, \Vert y \Vert \}$. Recall that a Banach space $X$ is said to be uniformly convex if for every $0 < \epsilon \leq 2$ there exists $\delta >0$ such that for any $x,y \in X$ with $\Vert x \Vert=1= \Vert y \Vert$ and $\Vert x-y \Vert  \geq \epsilon$ we have $\Big\Vert \frac{x+y}{2} \Big\Vert \leq 1- \delta$. It is known that every uniformly convex Banach space is reflexive. 

Indeed, best approximation results ensures the existence of approximate solutions but the solutions may not be optimal. On the other hand, best proximity point results yield optimal approximate solutions. Several authors (see \cite{1,2,3,5,6,9,11,13}) have obtained best proximity points of  various contraction and nonexpansive mappings in Banach spaces and metric spaces. The following result by  Kirk et al. \cite{5} guarantees that $A_0$ and $B_0$ are nonempty:
\begin{lemma}\label{lemma3}
\emph{\cite[Lemma 3.2]{5}} Let $A$ and $B$ be two nonempty, closed and convex subsets of a reflexive Banach space $X$. Suppose that $A$ is bounded. Then $A_0$ and $B_0$ are nonempty. 
\end{lemma}
Eldred and Veeramani \cite{1} introduced the notion of cyclic contraction mappings and established the existence of a best proximity point of cyclic contraction mapping.

Let $A$ and $B$ be two nonempty subsets of a metric space $(X,d)$. A mapping $T: A \cup B \rightarrow A \cup B$ is a cyclic contraction mapping if it satisfies the following conditions:

(i) $T(A) \subset B $ and $T(B) \subset A$,

(ii) $d(Tx,Ty) \leq \lambda d(x,y)+( 1-\lambda)dist(A,B)$, for some $\lambda \in (0,1)$ and all $x \in A$, $y \in B$.
\begin{theorem} \emph{\cite[Theorem 3.10]{1}}
Let $A$ and $B$ be nonempty, closed and convex subsets of a uniformly convex Banach space $X$. Let $T: A \cup B \rightarrow A \cup B$ be a cyclic contraction mapping. Then $T$ has a unique best proximity point.
\end{theorem}
Sankar and Veeramani \cite{9} introduced the notion of cyclic nonexpansive mappings and studied best proximity points of such mappings.

Let $A$ and $B$ be  nonempty subsets of a normed linear space $X$. A mapping $T:A \rightarrow B$ is a cyclic nonexpansive mapping if it satisfies the following conditions:

(i) $T(A) \subset B$ and $T(B) \subset A$,

(ii) $\Vert Tx-Ty \Vert \leq \Vert x-y \Vert$ for all $x \in  A$ and $y \in B$.
\begin{theorem}
\emph{\cite[Theorem 3.5]{9}} Let $A$ and $B$ be two  nonempty, closed and convex subsets of a uniformly convex Banach space $X$. Suppose that $A$ is bounded and $A_0$ is compact. Let $T:A \rightarrow B$ be a  cyclic nonexpansive mapping. Then $T$ has a best proximity point.
\end{theorem} 
Recently, many authors (see \cite{4,7,8,10}) have studied Ulam-Hyers stability for integral equations, differential equations, operatorial equations and various fixed point problems in different spaces. In this paper we extend the  notion of Ulam-Hyers stability to coupled best proximity point problem as follows:
\begin{definition}
Let $A$ and $B$ be two nonempty subsets of a Banach space $X$. Let $T:(A \times B) \cup (B \times A) \rightarrow A \cup B$ be a mapping. Then the coupled best proximity point problem is said to be Ulam-Hyers stable if and only if for each $\epsilon>0$ and each $(u,v) \in A \times B$ satisfy the following inequalities:
\begin{align*}
\Vert u-T(u,v) \Vert & \leq  \epsilon+dist(A,B),\\
\Vert v-T(v,u) \Vert & \leq \epsilon+dist(A,B)
\end{align*}
there exist $\alpha, \beta >0$ and a coupled best proximity point $(x^*,y^*)$ of $T$  such that
\begin{align*}
\Vert x^*-u \Vert &\leq \alpha \epsilon+ \beta dist(A,B),\\
\Vert y^*-v \Vert & \leq \alpha \epsilon+ \beta dist(A,B).
\end{align*}
\end{definition}

In this paper we introduce the notion of p-cyclic contraction and p-cyclic nonexpansive mappings. We extend the notion of cyclic contraction and cyclic nonexpansive mappings introduced by Eldred and Veeramani \cite{1} and Sankar and Veeramani \cite{9}, respectively. The main objective of the paper is to formulate necessary conditions which ensure the existence of a coupled best proximity point of such mappings in the setting of uniformly convex Banach spaces. The coupled best proximity point theorems obtained not only ascertain the existence of an approximate solution but also ensure its optimality.  Also, we investigate Ulam-Hyers stability of the coupled best proximity point problem in the case of p-cyclic contraction and p-cyclic nonexpansive mappings. An example is also provided to illustrate the efficiency of the results. 
\section{Main Results}
Throughout this section, we denote by $\mathbb{N}$ the set of natural numbers. We obtain coupled best proximity point results of p-cyclic contraction and p-cyclic nonexpansive mappings in the context of uniformly convex Banach spaces.
\begin{definition}
\emph{Let $A$ and $B$ be two nonempty subsets of a Banach space $X$. A mapping $T:(A \times B) \cup (B \times A) \rightarrow A \cup B$ is a} p-cyclic contraction mapping \emph{if it satisfies the following conditions:}

\emph{(i) $T(A,B) \subset B$ and $T(B,A) \subset A$,}

\emph{(ii) $\Vert T(x_1,y_1)-T(x_2,y_2) \Vert \leq  \lambda \Vert (x_1,y_1)-(x_2,y_2) \Vert +(1-\lambda)dist(A,B)$, for some $\lambda \in (0,1)$.}
\end{definition}
The following results are instrumental in proving the subsequent results:
\begin{proposition}\label{proposition1}
Let $A$ and $B$ be two nonempty subsets of a Banach space $X$. Let $T:(A \times B) \cup (B \times A) \rightarrow A \cup B$ be a p-cyclic contraction mapping. Then starting with any $(x_0,y_0) \in A \times B$ we have
$$\Vert (x_n,y_n)-(T(x_n,y_n),T(y_n,x_n)) \Vert \rightarrow dist(A,B),$$
where $x_n=T(x_{n-1},y_{n-1})$ and $y_n=T(y_{n-1},x_{n-1})$ for each $n \in \mathbb{N}$.
\end{proposition}
\begin{proof}
Consider 
\begin{align*}
\Vert x_n-T(x_n,y_n) \Vert &= \Vert T(x_{n-1},y_{n-1})-T(x_n,y_n) \Vert \\
& \leq \lambda \Vert (x_{n-1}-y_{n-1})-(x_n,y_n) \Vert+(1-\lambda) dist(A,B) 
\end{align*}
Similarly, $\Vert y_n-T(y_n,x_n) \Vert \leq \lambda \Vert (x_{n-1},y_{n-1})-(x_n,y_n) \Vert +(1-\lambda) dist (A,B)$. Therefore,
\begin{align*}
\Vert (x_n,y_n)-(T(x_n,y_n), T(y_n,x_n)) \Vert&= \max \{ \Vert x_n-T(x_n,y_n) \Vert, \Vert y_n-T(y_n,x_n) \Vert \}\\
 & \leq \lambda \Vert (x_{n-1},y_{n-1}),(x_n,y_n) \Vert +(1-\lambda) dist(A,B)\\
&\leq \lambda^2  \Vert (x_{n-2},y_{n-2})-(x_{n-1},y_{n-1}) \Vert+(1-\lambda)(1+\lambda) dist(A,B).
\end{align*}
Proceeding likewise we obtain
\begin{align*}
\Vert (x_n,y_n)-(T(x_n,y_n), T(y_n,x_n)) \Vert & \leq \lambda^n \Vert (x_0,y_0)-(x_1,y_1)\Vert  +(1-\lambda)(1+\lambda+\lambda^2+\\
& \quad \ldots +\lambda^{n-1}) dist(A,B)\\
& \leq  \lambda^n  \Vert (x_0,y_0)-(x_1,y_1) \Vert + dist(A,B)
\end{align*}
Letting $n \rightarrow \infty$ we get, $\Vert (x_n,y_n)-(T(x_n,y_n), T(y_n,x_n)) \Vert \rightarrow dist(A,B)$.
\end{proof}
\begin{proposition}\label{proposition2}
Let $A$ and $B$ be two nonempty and closed subsets of a Banach space $X$. Let $T:(A \times B) \cup (B \times A) \rightarrow A \cup B$ be a p-cyclic contraction mapping. Let $(x_0,y_0) \in A \times B$, $x_n=T(x_{n-1},y_{n-1})$ and $y_n=T(y_{n-1},x_{n-1})$ for each $n \in \mathbb{N}$.  Suppose that $\{(x_{2n},y_{2n}) \}$ has a convergent subsequence in $A \times B$. Then there exists $(x,y) \in A \times B$ such that
\begin{align*}
\Vert x-T(x,y) \Vert & =dist(A,B),\\
\Vert y-T(y,x) \Vert &=dist(A,B).
\end{align*}
\end{proposition}
\begin{proof}
Suppose that $\{(x_{2n_i},y_{2n_i}) \}$ be  a subsequence of $\{(x_{2n},y_{2n}) \}$ converging to $(x,y) \in A \times B$. Consider
\begin{align*}
dist(A,B)& \leq \max \{ \Vert x-x_{2n_i-1} \Vert, \Vert y-y_{2n_i-1} \Vert \}\\
&= \Vert (x,y)-(x_{2n_i-1},y_{2n_i-1}) \Vert \\
& \leq \Vert (x,y)- (x_{2n_i},y_{2n_i}) \Vert +\Vert (x_{2n_i},y_{2n_i})-(x_{2n_i-1},y_{2n_i-1}) \Vert.
\end{align*}
Using Proposition \ref{proposition1} we deduce that
\begin{equation}\label{equation1}
\Vert (x,y)-(x_{2n_i-1},y_{2n_i-1}) \Vert \rightarrow dist(A,B).
\end{equation}
 Consider
\begin{align*}
dist(A,B)& \leq \Vert (x_{2n_i},y_{2n_i})-(T(x,y),T(y,x)) \Vert \\
&= \Vert (T(x_{2n_i-1},y_{2n_i-1}),T(y_{2n_i-1},x_{2n_i-1}))-(T(x,y),T(y,x)) \Vert \\
& = \max \{ \Vert T(x_{2n_i-1},y_{2n_i-1})-T(x,y) \Vert, \Vert T(y_{2n_i-1},x_{2n_i-1})-T(y,x) \Vert \}\\
&\leq \lambda  \Vert (x_{2n_i-1}, y_{2n_i-1})-(x,y) \Vert +(1- \lambda) dist(A,B).
\end{align*}
Using (\ref{equation1}) we get $\lim\limits_{k \rightarrow \infty} \Vert (x_{2n_i},y_{2n_i})-(T(x,y),T(y,x))\Vert = dist(A,B)$. Since $\Vert. \Vert$ is continuous, $ \Vert (x,y)-(T(x,y),T(y,x)) \Vert = dist(A,B)$. Now 
\begin{align*}
dist(A,B) \leq \Vert x-T(x,y) \Vert& \leq \max \{ \Vert x-T(x,y) \Vert, \Vert y-T(y,x) \Vert \}\\
&= \Vert (x,y)-(T(x,y),T(y,x)) \Vert\\
& =dist(A,B).
\end{align*}
Therefore, $\Vert x-T(x,y) \Vert =dist(A,B)$. Similarly, $\Vert y-T(y,x) \Vert = dist(A,B)$. 
\end{proof}
 A parallel result  to \cite[Lemma 3.7]{1} can be obtained in the following form:
\begin{lemma}\label{lemma1}
Let $A$ and $B$ be two nonempty, closed and convex subsets of a uniformly convex Banach space $X$. Let $\{(x_n,y_n) \}$ and $\{(u_n,v_n) \}$ be sequences in $A \times B$ and $\{(w_n,z_n) \}$ be a sequence in $B \times A$ satisfying

(i) $\Vert (x_n,y_n)-(w_n,z_n) \Vert \rightarrow dist(A,B)$,

(ii) for each $\epsilon>0$ there exists $N \in \mathbb{N}$ such that for all $m>n \geq N$ we have
$$\Vert (u_m,v_m)-(w_n,z_n) \Vert \leq dist(A,B)+\epsilon.$$
Then for each $\epsilon>0$ there exists $N_0 \in \mathbb{N}$ such that for all $m>n \geq N_0$ we have
$$\Vert (u_m,v_m)-(x_n,y_n) \Vert \leq \epsilon.$$
\end{lemma}
\begin{lemma}\label{lemma2}
Let $A$ and $B$ be two nonempty, closed and convex subsets of a uniformly convex Banach space $X$. Let $\{(x_n,y_n) \}$ and $\{(u_n,v_n) \}$ be sequences in $A \times B$ and $\{(w_n,z_n) \}$ be a sequence in $B \times A$ satisfying

(i) $\Vert (x_n,y_n)- (w_n,z_n) \Vert \rightarrow dist(A,B)$,

(ii) $\Vert (u_n,v_n)-(w_n,z_n) \Vert \rightarrow dist(A,B)$.\\
Then $\Vert (x_n,y_n) -(u_n,v_n) \Vert \rightarrow 0$.
\end{lemma}
\begin{theorem}\label{theorem1}
Let $A$ and $B$ be two nonempty, closed and convex subsets of a uniformly convex Banach space $X$. Let $T:(A \times B) \cup (B \times A) \rightarrow A \cup B$ be a p-cyclic contraction mapping. Then $T$ has a unique coupled best proximity point. Moreover, the coupled best proximity point problem is Ulam-Hyers stable.
\end{theorem}
\begin{proof}
By Proposition \ref{proposition1} we have
\begin{equation} \label{equation2}
\Vert (x_{2n},y_{2n})-(T(x_{2n},y_{2n}),T(y_{2n},x_{2n})) \Vert \rightarrow dist(A,B).
\end{equation}
We show that $\Vert (T(x_{2n+1},y_{2n+1}),T(y_{2n+1},x_{2n+1}))-(T(x_{2n},y_{2n}),T(y_{2n},x_{2n})) \Vert \rightarrow dist(A,B)$. Consider
\begin{align*}
dist(A,B) & \leq \Vert (T(x_{2n+1},y_{2n+1}),T(y_{2n+1},x_{2n+1}))- (T(x_{2n},y_{2n}),T(y_{2n},x_{2n})) \Vert\\
 &= \max \{ \Vert T(x_{2n+1},y_{2n+1})-T(x_{2n},y_{2n}) \Vert,  \Vert T(y_{2n+1},x_{2n+1})-T(y_{2n},x_{2n}) \Vert \}\\
& \leq \lambda \Vert (x_{2n+1},y_{2n+1})-(x_{2n},y_{2n}) \Vert +(1- \lambda) dist(A,B)\\
&= \lambda \Vert (T(x_{2n},y_{2n}),T(y_{2n},x_{2n}))-(x_{2n},y_{2n}) \Vert +(1-\lambda) dist(A,B)
\end{align*}
which gives \begin{equation} \label{equation3}
 \Vert (T(x_{2n+1},y_{2n+1}),T(y_{2n+1},x_{2n+1}))-(T(x_{2n},y_{2n}),T(y_{2n},x_{2n})) \Vert  \rightarrow dist(A,B).
\end{equation}
Using (\ref{equation2}), (\ref{equation3}) and Lemma \ref{lemma2} we deduce that 
\begin{equation}\label{equation4}
\Vert (x_{2n},y_{2n})- (T(x_{2n+1},y_{2n+1}),T(y_{2n+1},x_{2n+1})) \Vert \rightarrow 0.
\end{equation}
Similarly, we show that
\begin{equation}\label{equation5}
\Vert (x_{2n+1},y_{2n+1})- (T(x_{2n+2},y_{2n+2}),T(y_{2n+2},x_{2n+2})) \Vert \rightarrow 0.
\end{equation}
Now we show that for each $\epsilon>0$ there exists $N \in \mathbb{N}$ such that for all $m> n \geq N$ we have
$$\Vert (x_{2m},y_{2m})-(T(x_{2n},y_{2n}),T(y_{2n},x_{2n})) \Vert \leq dist(A,B)+ \epsilon.$$
Assume on the contrary, there exists $\epsilon_0>0$ such that for all $i \in \mathbb{N}$, there exists $m_i>n_i \geq i$ for which
$$\Vert (x_{2m_i},y_{2m_i})-(T(x_{2n_i},y_{2n_i}),T(y_{2n_i},x_{2n_i})) \Vert > dist(A,B)+ \epsilon_0.$$
Let $m_i$ be the least positive integer satisfying this  inequality, i.e., 
$$\Vert (x_{2(m_i-1)},y_{2(m_i-1)})-(T(x_{2n_i},y_{2n_i}),T(y_{2n_i},x_{2n_i})) \Vert \leq dist(A,B)+ \epsilon_0.$$
Consider
\begin{align*}
dist(A,B)+ \epsilon_0 & < \Vert (x_{2m_i},y_{2m_i})-(T(x_{2n_i},y_{2n_i}),T(y_{2n_i},x_{2n_i})) \Vert \\
& \leq \Vert (x_{2m_i},y_{2m_i})-(x_{2(m_i-1)},y_{2(m_i-1)}) \Vert+ \Vert (x_{2(m_i-1)},y_{2(m_i-1)})-(T(x_{2n_i},y_{2n_i}),T(y_{2n_i},x_{2n_i})) \Vert \\
& \leq \Vert (x_{2m_i},y_{2m_i})-(x_{2m_i-2},y_{2m_i-2}) \Vert +dist(A,B)+ \epsilon_0.
\end{align*}
Letting $i \rightarrow \infty$ and using (\ref{equation4}) we have 
\begin{equation}\label{equation6}
\lim\limits_{i \rightarrow \infty} \Vert (x_{2m_i},y_{2m_i})-(T(x_{2n_i},y_{2n_i}),T(y_{2n_i},x_{2n_i})) \Vert = dist(A,B) + \epsilon_0.
\end{equation}
Consider
\begin{align*}
\Vert (x_{2m_i},y_{2m_i})-(T(x_{2n_i},y_{2n_i}),T(y_{2n_i},x_{2n_i}))\Vert & \leq \Vert (x_{2m_i},y_{2m_i})-(T(x_{2m_i+1},y_{2m_i+1}),T(y_{2m_i+1},x_{2m_i+1}))\Vert \\
& \quad + \Vert (T(x_{2m_i+1},y_{2m_i+1}),(T(y_{2m_i+1},x_{2m_i+1}))-(T(x_{2n_i+2},\\
& \quad y_{2n_i+2}),T(y_{2n_i+2},x_{2n_i+2})) \Vert + \Vert (T(x_{2n_i+2},y_{2n_i+2}),\\
& \quad T(y_{2n_i+2},x_{2n_i+2}))-(T(x_{2n_i},y_{2n_i}),T(y_{2n_i},x_{2n_i})) \Vert
\end{align*}
Also, observe that 
\begin{align*}
\Vert (T(x_{2m_i+1},&y_{2m_i+1}),(T(y_{2m_i+1},x_{2m_i+1}))-(T(x_{2n_i+2}, y_{2n_i+2}),T(y_{2n_i+2},x_{2n_i+2})) \Vert\\
&= \max \{ \Vert T(x_{2m_i+1},y_{2m_i+1})-T(x_{2n_i+2},y_{2n_i+2}) \Vert, \Vert T(y_{2m_i+1},x_{2m_i+1})-T(y_{2n_i+2},x_{2n_i+2}) \Vert \}\\
& \leq \lambda  \Vert (x_{2m_i+1},y_{2m_i+1})-(x_{2n_i+2},y_{2n_i+2}) \Vert+(1- \lambda)dist(A,B)\\
& \leq \lambda^2 \Vert (x_{2m_i},y_{2m_i})-(x_{2n_i+1},y_{2n_i+1}) \Vert +(1-\lambda^2)dist(A,B).
\end{align*}
Therefore, 
\begin{align*}
dist(A,B)+ \epsilon_0 & \leq \Vert (x_{2m_i},y_{2m_i})-(T(x_{2n_i},y_{2n_i}),T(y_{2n_i},x_{2n_i})) \Vert \\
& \leq \Vert (x_{2m_i},y_{2m_i})- (T(x_{2m_i+1},y_{2m_i+1}),T(y_{2m_i+1},x_{2m_i+1})) \Vert \\
& \quad + \lambda^2 \Vert (x_{2m_i},y_{2m_i})-(x_{2n_i+1},y_{2n_i+1}) \Vert +(1-\lambda^2)dist(A,B)\\
& \quad + \Vert (T(x_{2n_i+2},y_{2n_i+2}), T(y_{2n_i+2},x_{2n_i+2}))-(T(x_{2n_i},y_{2n_i}),T(y_{2n_i},x_{2n_i})) \Vert.
\end{align*}
Using (\ref{equation4}), (\ref{equation5}) and (\ref{equation6}) we have $dist(A,B)+ \epsilon_0 \leq dist(A,B)+ \lambda^2 \epsilon$ which implies that $\lambda^2 \geq 1$, a contradiction. Therefore, for every $\epsilon >0$ there exists $N \in \mathbb{N}$ such that for all $m>n \geq N$ we have
$$\Vert (x_{2m},y_{2m})-(T(x_{2n},y_{2n}),T(y_{2n},x_{2n})) \Vert \leq dist(A,B)+ \epsilon.$$
Then using (\ref{equation2}) and Lemma \ref{lemma1}, for every $\epsilon>0$ there exists $N_0 \in \mathbb{N}$ such that for all $m>n \geq N_0$ we have
$$\Vert (x_{2m},y_{2m})-(x_{2n},y_{2n}) \Vert \leq \epsilon.$$
Therefore, $\{(x_{2n},y_{2n})\}$ is a Cauchy sequence and hence convergent. By Proposition \ref{proposition2} there exists $(x,y) \in A \times B$ such that
\begin{align*}
\Vert x-T(x,y) \Vert & =dist(A,B),\\
\Vert y-T(y,x) \Vert &=dist(A,B).
\end{align*}
Therefore, 
\begin{equation}\label{equation10}
\Vert (x,y)-(T(x,y),T(y,x))\Vert = \max \{ \Vert x-T(x,y) \Vert, \Vert y-T(y,x) \Vert \}=dist(A,B).
\end{equation}
Let $(\overline{x},\overline{y})$ be another coupled best proximity point of $T$. Consider 
\begin{align*}
\Vert T(T(x,y),T(y,x))-T(x,y) \Vert & \leq \lambda \Vert (T(x,y), T(y,x))-(x,y) \Vert +(1-\lambda)dist(A,B)\\
& =dist(A,B).
\end{align*}
Similarly, $\Vert T(T(y,x),T(x,y))-T(y,x) \Vert=dist(A,B)$. This implies that
\begin{align*}
\Vert (T(T(x,y),T(y,x)),T(T(y,x),T(x,y)))-(T(x,y),T(y,x))\Vert &= \max \{ \Vert T(T(x,y),T(y,x))-T(x,y) \Vert, \\
& \quad \Vert T(T(y,x),T(x,y))-T(y,x))\Vert \}\\
&= dist(A,B).
\end{align*}
Using (\ref{equation10}) and Lemma \ref{lemma2} we get $\Vert (x,y)-(T(T(x,y),T(y,x)),T(T(y,x),T(x,y))) \Vert =0$. Therefore, 
$$T(T(x,y),T(y,x))=x \quad \mbox{and} \quad T(T(y,x),T(x,y))=y.$$
Similarly,
$$T(T(\overline{x},\overline{y}),T(\overline{y},\overline{x}))=\overline{x} \quad \mbox{and} \quad T(T(\overline{y},\overline{x}),T(\overline{x},\overline{y}))=\overline{y}.$$
Consider 
\begin{align*}
\Vert T(x,y)-\overline{x} \Vert &=\Vert T(x,y)-T(T(\overline{x},\overline{y}),T(\overline{y},\overline{x}))\Vert \\
& \leq \lambda  \Vert (x,y)-(T(\overline{x},\overline{y}),T(\overline{y},\overline{x}) )\Vert +(1- \lambda)dist(A,B)\\
& \leq \lambda \Vert (x,y)-(T(\overline{x},\overline{y}),T(\overline{y},\overline{x}) )\Vert+(1-\lambda) \max \{ \Vert x-T(\overline{x},\overline{y}) \Vert, \Vert y-T(\overline{y}, \overline{x})\Vert \}\\
&= \Vert (x,y)-(T(\overline{x},\overline{y}),T(\overline{y},\overline{x}) )\Vert.
\end{align*}
Similarly,
$$\Vert T(y,x)- \overline{y} \Vert \leq \Vert (x,y)-(T(\overline{x},\overline{y}),T(\overline{y},\overline{x}) )\Vert.$$
This gives
\begin{equation}\label{equation7}
\Vert (T(x,y),T(y,x))- (\overline{x},\overline{y}) \Vert \leq \Vert (x,y)-(T(\overline{x},\overline{y}),T(\overline{y},\overline{x})) \Vert.
\end{equation}
Similarly,
\begin{equation}\label{equation8}
\Vert (T(\overline{x},\overline{y}),T(\overline{y},\overline{x}))-(x,y) \Vert \leq \Vert (\overline{x},\overline{y}),(T(x,y),T(y,x))\Vert.
\end{equation}
From (\ref{equation7}) and (\ref{equation8}) we conclude that
\begin{equation}\label{equation9}
\Vert (x,y)-(T(\overline{x},\overline{y}),T(\overline{y},\overline{x})) \Vert = \Vert (\overline{x},\overline{y})-(T(x,y),T(y,x))\Vert.
\end{equation}
Observe that $\Vert (x,y)-(T(\overline{x},\overline{y}),T(\overline{y},\overline{x}))\Vert >dist(A,B)$ because if $\Vert (x,y)-(T(\overline{x},\overline{y}),T(\overline{y},\overline{x}))\Vert =dist(A,B)$, then using (\ref{equation10}), (\ref{equation9}) and Lemma \ref{lemma2} we deduce that $(x,y)=(\overline{x},\overline{y})$. Similarly, $\Vert (\overline{x},\overline{y}) -(T(x,y),T(y,x))\Vert >dist(A,B)$. Consider
\begin{align*}
\Vert (\overline{x},\overline{y})-(T(x,y),T(y,x))\Vert &=\max \{ \Vert \overline{x}-T(x,y) \Vert, \Vert \overline{y}-T(y, x) \Vert \}\\
& = \max \{ \Vert T(T(\overline{x},\overline{y}),T(\overline{y}, \overline{x}))-T(x,y) \Vert, \Vert T(T(\overline{y},\overline{x}),T(\overline{x},\overline{y}))-T(y,x) \Vert \}\\
& \leq \lambda  \Vert (x,y)-(T(\overline{x},\overline{y},T(\overline{y}, \overline{x})) \Vert+(1-\lambda)dist(A,B)\\
&< \lambda \Vert (x,y)-(T(\overline{x}, \overline{y}),T(\overline{y},\overline{x})) \Vert +(1- \lambda ) \Vert (x,y)-(T(\overline{x},\overline{y}),T(\overline{y},\overline{x}))\Vert.
\end{align*}
This implies that $\Vert (\overline{x},\overline{y})-(T(x,y),T(y,x)) \Vert < \Vert (x,y)-(T(\overline{x},\overline{y}),T(\overline{y}, \overline{x})) \Vert$, a contradiction. Therefore, $(x,y)=(\overline{x},\overline{y})$. Hence, $T$ has a unique coupled best proximity point. Let $\epsilon>0$ be given and $(u,v) \in A \times B$ satisfy \begin{align*}
\Vert u-T(u,v )\Vert & \leq \epsilon+dist(A,B),\\
\Vert v-T(v,u) \Vert & \leq \epsilon+dist(A,B).
\end{align*}
Consider 
\begin{align*}
\Vert x-u \Vert & \leq \Vert x-T(x,y) \Vert +\Vert T(x,y)- T(u,v) \Vert + \Vert T(u,v)-u\Vert \\
& \leq dist(A,B)+ \lambda  \Vert (x,y)-(u,v) \Vert+(1-\lambda) dist(A,B)+ \epsilon+dist(A,B)\\
&= \epsilon+(3-\lambda) dist(A,B)+ \lambda \Vert (x,y)-(u,v) \Vert.
\end{align*}
Similarly, 
$$\Vert y-v \Vert \leq \epsilon+(3-\lambda) dist(A,B)+\lambda \Vert (x,y)-(u,v) \Vert.$$
Therefore, $\max \{ \Vert x-u \Vert, \Vert y-v \Vert \} \leq \epsilon+(3-\lambda) dist(A,B)+\lambda \max \{ \Vert x-u \Vert, \Vert y-v \Vert \}$ which implies that
\begin{align*}
\Vert x-u \Vert & \leq \frac{\epsilon}{1-\lambda}+\Big( \frac{3-\lambda}{1-\lambda} \Big) dist(A,B),\\
\Vert y-v \Vert & \leq \frac{\epsilon}{1-\lambda}+ \Big( \frac{3-\lambda}{1-\lambda} \Big) dist(A,B).
\end{align*}
Hence, the coupled best proximity problem is Ulam-Hyers stable.
\end{proof}
\begin{definition}
\emph{Let $A$ and $B$ be two nonempty subsets of a Banach space $X$. A mapping $S:(A \times B) \cup (B \times A) \rightarrow A \cup B$ is a} p-cyclic nonexpansive mapping \emph{if it satisfies the following conditions:}

\emph{(i) $S(A,B) \subset B \quad \mbox{and} \quad S(B,A) \subset A$,}

\emph{(ii) $\Vert S(x_1,y_1)-S(x_2,y_2) \Vert \leq  \Vert (x_1,y_1)-(x_2,y_2) \Vert $.}
\end{definition}
\begin{theorem}\label{theorem2}
Let $A$ and $B$ be two nonempty, closed, bounded and convex subsets of a uniformly convex Banach space $X$. Let $S:(A \times B) \cup (B \times A) \rightarrow A \cup B$ be  a p-cyclic nonexpansive mapping. Then there exist a sequence $\{(x_n,y_n) \}$ in $A_0 \times B_0$ and $(x^*,y^*) \in A_0 \times B_0$ satisfying the following conditions:

(i) $(x_n,y_n) \rightharpoonup (x^*,y^*)$, where $\rightharpoonup$ denotes weak convergence,

(ii) $\Vert (x^*,y^*)-(S(x^*,y^*),S(y^*,x^*)) \Vert \leq dist(A,B)+ \liminf\limits_{n \rightarrow \infty} \Vert (x_n,y_n)-(x^*,y^*) \Vert$.
\end{theorem}
\begin{proof}
Since every uniformly convex Banach space is reflexive, by Lemma \ref{lemma3} we infer that $A_0$ is nonempty. Then there exists $(x_0,y_0) \in A_0 \times B_0$ such that $\Vert x_0-y_0 \Vert =dist(A,B)$. For each $n \in \mathbb{N}$, define $T_n:(A \times B) \cup (B \times A) \rightarrow A \cup B$ by
$$T_n(x,y):=\left\{\begin{array}{cc}
\frac{1}{n}S(x_0,y_0)+ \Big(1-\frac{1}{n} \Big) S(x,y),& \mbox{if}\thinspace \thinspace (x,y) \in A \times B,\\
\frac{1}{n}S(y_0,x_0)+\Big(1- \frac{1}{n} \big) S(y,x), &\mbox{if} \thinspace \thinspace (x,y) \in B \times A.
\end{array}
\right. $$
Since $A$ and $B$ are convex, $T_n(A,B) \subset B$ and $T_n(B,A) \subset A$ for each $n \in \mathbb{N}$. If $(x_1,y_1),(x_2,y_2) \in A \times B$, then consider 
\begin{align*}
\Vert T_n(x_1,y_1)-T_n(x_2,y_2) \Vert & = \Big(1-\frac{1}{n} \Big) \Vert  S(x_1,y_1)-S(x_2,y_2) \Vert \\
& \leq \Big(1-\frac{1}{n} \Big)  \Vert (x_1,y_1)-(x_2,y_2) \Vert\\
& \leq \Big(1-\frac{1}{n} \Big) \Vert (x_1,y_1)-(x_2,y_2) \Vert + \frac{1}{n} dist(A,B).
\end{align*}
If $(x_1,y_1),(x_2,y_2) \in B \times A$, then we proceed as in the previous case. If $(x_1,y_1) \in A \times B$ and $(x_2,y_2) \in B \times A$, then consider
\begin{align*}
\Vert T_n(x_1,y_1)-T_n(x_2,y_2) \Vert & \leq \frac{1}{n} \Vert S(x_0,y_0)-S(y_0,x_0) \Vert + \Big( 1- \frac{1}{n} \Big) \Vert S(x_1,y_1)-S(x_2,y_2) \Vert \\
& \leq \frac{1}{n} \Vert x_0-y_0 \Vert + \Big( 1- \frac{1}{n} \Big) \Vert (x_1,y_1)-(x_2,y_2) \Vert\\
&=\frac{1}{n} dist(A,B) + \Big( 1- \frac{1}{n} \Big) \Vert (x_1,y_1)-(x_2,y_2) \Vert .
\end{align*}
This implies that $T_n$ is a p-cyclic contraction mapping. Therefore, by Theorem \ref{theorem1} there exists unique $(x_n,y_n) \in A \times B$ such that
\begin{align*}
\Vert x_n-T_n(x_n,y_n) \Vert& = dist(A,B),\\
\Vert y_n-T_n(y_n,x_n) \Vert &=dist(A,B).
\end{align*}
As $T_n(x_n,y_n) \in B$ and $T_n(y_n,x_n) \in A$, $(x_n,y_n) \in A_0 \times B_0$. Since $A_0 \times B_0$ is bounded, closed and convex, $\{(x_n,y_n) \}$ has a weakly convergent subsequence. Assume that the sequence $\{(x_n,y_n)\}$ itself converges weakly to $(x^*,y^*) \in A_0 \times B_0$. Therefore, 
$$(x_n,y_n)-(S(x^*,y^*),S(y^*,x^*)) \rightharpoonup (x^*,y^*)-(S(x^*,y^*),S(y^*,x^*)).$$
As $\Vert . \Vert$ is weakly lower semi-continuous,
\begin{align*}
\Vert (x^*,y^*)-(S(x^*,y^*),S(y^*,x^*)) \vert & \leq \liminf\limits_{n \rightarrow \infty} \Vert (x_n,y_n)-(S(x^*,y^*),S(y^*,x^*)) \Vert \\
& \leq \liminf\limits_{n \rightarrow \infty} \{ \Vert (x_n,y_n)-(T_n(x_n,y_n),T_n(y_n,x_n)) \Vert  \\
& \quad +\Vert (T_n(x_n,y_n),T_n(y_n,x_n))-(S(x_n,y_n),S(y_n,x_n)) \Vert\\
& \quad  + \Vert S(x_n,y_n),S(y_n,x_n))-(S(x^*,y^*),S(y^*,x^*)) \Vert \}
\end{align*}
Since $(x_n,y_n)$ is a coupled best proximity point of $T_n$, 
$$\Vert(x_n,y_n)-(T_n(x_n,y_n),T_n(y_n,x_n)) \Vert = \max \{ \Vert x_n-T_n(x_n,y_n) \Vert, \Vert y_n-T_n(y_n,x_n) \Vert \}=dist(A,B).$$
Consider 
\begin{align*}
\Vert T_n(x_n,y_n)-S(x_n,y_n) \Vert & = \Big\Vert \frac{1}{n} S(x_0,y_0)+ \Big(1- \frac{1}{n} \Big) S(x_n,y_n)-S(x_n,y_n) \Big\Vert \\
& = \frac{1}{n} \Vert S(x_0,y_0)-S(x_n,y_n) \Vert \\
& \leq \frac{1}{n} \Vert (x_0,y_0)-(x_n,y_n) \Vert.
\end{align*}
Similarly, 
\begin{align*}
\Vert T_n(y_n,x_n)-S(y_n,x_n) \Vert & \leq \frac{1}{n} \Vert (x_0,y_0)-(x_n,y_n) \Vert.
\end{align*}
This gives 
\begin{align*}
\Vert (T_n(x_n,y_n),T_n(y_n,x_n))-(S(x_n,y_n),S(y_n,x_n))\Vert &= \max \{ \Vert T_n(x_n,y_n)-S(x_n,y_n) \Vert, \Vert T_n(y_n,x_n)-S(y_n,x_n) \Vert \}\\
& \leq \frac{1}{n} \max \{\Vert x_0-x_n \Vert, \Vert y_0-y_n \Vert \} \rightarrow 0 \thinspace \thinspace \mbox{as} \thinspace \thinspace n \rightarrow \infty.
\end{align*}
Since $S$ is p-cyclic nonexpansive mapping, 
\begin{align*}
\Vert (S(x_n,y_n),S(y_n,x_n))-(S(x^*,y^*),S(y^*,x^*)) \Vert &= \max \{ \Vert S(x_n,y_n)-S(x^*,y^*) \Vert, \Vert S(y_n,x_n)-S(y^*,x^*) \Vert \} \\
& \leq  \Vert (x_n,y_n)-(x^*,y^*) \Vert.
\end{align*}
Therefore, 
\begin{align*}
\Vert (x^*,y^*)-(S(x^*,y^*),S(y^*,x^*)) \Vert & \leq \liminf\limits_{n \rightarrow \infty} \Big\{ dist(A,B)+ \frac{1}{n} \Vert (x_0,y_0)-(x_n,y_n) \Vert\\
& \quad + \Vert (x_n,y_n)-(x^*,y^*) \Vert \Big\}.
\end{align*}
Hence, $\Vert (x^*,y^*)-(S(x^*,y^*),S(y^*,x^*)) \Vert \leq dist(A,B)+ \liminf\limits_{n \rightarrow \infty} \Vert (x_n,y_n)-(x^*,y^*) \Vert $.
\end{proof}
\begin{theorem}\label{theorem3}
Let $A$ and $B$ be two nonempty, closed, bounded and convex subsets of a uniformly convex Banach space $X$ such that $A_0 \times B_0$ is compact. Let $S:(A \times B) \cup (B \times A) \rightarrow A \cup B$ be a p-cyclic nonexpansive mapping. Then $S$ has a coupled best proximity point. Moreover, the coupled  best proximity point problem is Ulam-Hyers stable.
\end{theorem}
\begin{proof}
By Theorem \ref{theorem2} there exists a sequence $\{(x_n,y_n)\}$ in $A_0 \times B_0$ and $(x^*,y^*) \in A_0 \times B_0$ such that the following holds:

(i) $(x_n,y_n) \rightharpoonup (x^*,y^*)$,

(ii) $\Vert (x^*,y^*)-(S(x^*,y^*),S(y^*,x^*)) \Vert \leq dist(A,B)+ \liminf\limits_{n \rightarrow \infty} \Vert (x_n,y_n)-(x^*,y^*) \Vert$.\\
Since $A_0 \times B_0$ is compact, weak convergence implies strong convergence. Therefore, $\lim\limits_{n \rightarrow \infty} \Vert (x_n,y_n)-(x^*,y^*) \Vert=0$ which gives 
$$\Vert (x^*,y^*)-(S(x^*,y^*),S(y^*,x^*))\Vert \leq dist(A,B).$$
Consider $$dist(A,B) \leq \Vert x^*-S(x^*,y^*) \Vert \leq \Vert (x^*,y^*)-(S(x^*,y^*),S(y^*,x^*))\Vert \leq dist(A,B).$$ Therefore, $\Vert x^*-S(x^*,y^*) \Vert=dist(A,B)$. Similarly, $\Vert y^*-S(y^*,x^*) \Vert =dist(A,B)$. Since $A$ and $B$ are bounded, there exist $M_A,M_B>0$ such that $\Vert x \Vert \leq M_A$ for all $x \in A$ and $\Vert y \Vert \leq M_B$ for all $y \in B$. Let $\epsilon>0$ be given and $(u^*,v^*) \in A \times B$ satisfy
\begin{align*}
\Vert u^*-S(u^*,v^*) \Vert & \leq \epsilon+dist(A,B),\\
\Vert v^*-S(v^*,u^*) \Vert & \leq \epsilon+dist(A,B).
\end{align*} 
Consider 
\begin{align*}
\Vert x^*-u^* \Vert & \leq \Vert x^* \Vert + \Vert u^*-S(u^*,v^*) \Vert + \Vert S(u^*,v^*) \Vert \\
& \leq M_A+\epsilon+dist(A,B)+M_B\\
&= \epsilon+dist(A,B) \Big(1+\frac{M_A+M_B}{dist(A,B)} \Big).
\end{align*}
Similarly, 
$$ \Vert y^*-v^* \Vert \leq \epsilon+dist(A,B) \Big(1+\frac{M_A+M_B}{dist(A,B)} \Big)$$
which implies that the coupled best proximity point problem is Ulam-Hyers stable.  
\end{proof}
It is observed that the  coupled best proximity point obtained in Theorem \ref{theorem3} is not necessarily unique. This can be illustrated by the following example: 
\begin{example}
\emph{Let $X=\mathbb{R}^2$ and define $\Vert (x,y) \Vert= \sqrt{x^2+y^2}$. Then $X$ is a uniformly convex Banach space. Let $A=[0,1] \times [0,1]$ and $B=[0,1] \times [2,3]$. Then $A$ and $B$ are nonempty, closed, bounded and convex subsets of $X$. Also, $dist(A,B)=1$ and
\begin{align*}
A_0&=\{ (x,1): 0 \leq x \leq 1 \}\\
B_0&=\{ (x,2):0 \leq x \leq  1 \}.
\end{align*}
Since $A_0$ and $B_0$ are compact, $A_0 \times B_0$ is compact. Define $S:(A \times B) \cup (B \times A) \rightarrow A \cup B$ by
$$S((x,y),(u,v)):=(\sin u,v).$$ It is easily seen that $S(A,B) \subset B$ and $S(B,A) \subset A$. Consider
\begin{align*}
\Vert S((x_1,y_1),(u_1,v_1))-S((x_2,y_2),(u_2,v_2)) \Vert &= \Vert (\sin u_1,v_1)-(\sin u_2,v_2) \Vert \\
& = \sqrt{(\sin u_1-\sin u_2)^2+(v_1-v_2)^2}\\
& \leq \sqrt{(u_1-u_2)^2+(v_1-v_2)^2}\\
& \leq \max \{ \sqrt{(x_1-x_2)^2+(y_1-y_2)^2},\\
& \quad \sqrt{(u_1-u_2)^2+(v_1-v_2)^2}\}\\
&= \max \{ \Vert (x_1-x_2,y_1-y_2) \Vert, \Vert (u_1-u_2,v_1-v_2) \Vert \}\\
&= \max \{ \Vert (x_1,y_1)-(x_2,y_2) \Vert, \Vert (u_1,v_1)-(u_2,v_2) \Vert \}.
\end{align*} 
This implies that $S$ is a p-cyclic nonexpansive mapping. Therefore, all the conditions of Theorem \ref{theorem3} are satisfied which yields the existence of a coupled best proximity point of $S$. We observe that
\begin{align*}
\Vert (0,1)-S((0,1),(0,2)) \Vert&= \Vert (0,1)-(0,2) \Vert=1=dist(A,B),\\
\Vert (0,2)-S((0,2),(0,1)) \Vert &= \Vert (0,2)-(0,1) \Vert=1=dist(A,B).
\end{align*}
Therefore, $((0,1),(0,2))$ is a coupled best proximity point of $S$. Also, if we consider $x \in [0,1]$ such that $\sin x=x$, then  $((x,1),(x,2))$ is a coupled  best proximity point of $S$. }
\end{example}
The following result establishes the existence of a coupled  best proximity point and Ulam-Hyers stability in strictly convex Banach spaces.
\begin{theorem}\label{theorem4}
Let $A$ and $B$ be two nonempty, closed and convex subsets of a strictly convex Banach space $X$. Suppose that $A_0$ and $B_0$ are nonempty and $A_0 \times B_0$ is compact. Let $S:(A \times B) \cup (B \times A) \rightarrow A \cup B$ be a p-cyclic nonexpansive mapping. Then $S$ has a coupled  best proximity point. Additionally, assume that if $(x,y) \in A \times B$ is a coupled best proximity point  of $S$, then for every $(u,v) \in A \times B$ we have
$$\Vert x-T(u,v) \Vert\leq \Vert u-T(u,v) \Vert \quad \mbox{and} \quad \Vert y-T(v,u) \Vert \leq \Vert v-T(v,u) \Vert.$$ Then the coupled  best proximity point problem is Ulam-Hyers stable.
\end{theorem}
\begin{proof}
Since $A_0$ and $B_0$ are nonempty, suppose that $x_0 \in A_0$ and $y_0 \in B_0$. For each $n \in \mathbb{N}$ define the mapping $T_n$ same as in Theorem \ref{theorem2}. Since $A_0 \times B_0$ is compact, by Proposition \ref{proposition2} there exists $(x_n,y_n) \in A \times B$ such that
\begin{align*}
\Vert x_n-T(x_n,y_n) \Vert &= dist(A,B),\\
\Vert y_n-T(y_n,x_n) \Vert &=dist(A,B)
\end{align*} 
which implies that $(x_n,y_n) \in A_0 \times B_0$. As $A_0 \times B_0$ is compact, $\{(x_n,y_n)\}$ has a subsequence $\{(x_{n_i},y_{n_i})\}$ converging to some $(x^*,y^*) \in A_0 \times B_0$. Consider
\begin{align*}
\Vert (x^*,y^*)-(S(x^*,y^*),S(y^*,x^*))\Vert & \leq \Vert (x^*,y^*)-(x_{n_i},y_{n_i})\Vert+ \Vert (x_{n_i},y_{n_i})-(T_{n_i}(x_{n_i},y_{n_i}),T_{n_i}(y_{n_i},x_{n_i}))\Vert \\
& \quad +\Vert (T_{n_i}(x_{n_i},y_{n_i}),T_{n_i}(y_{n_i},x_{n_i}))-(S(x_{n_i},y_{n_i}),S(y_{n_i},x_{n_i}))\Vert\\
& \quad + \Vert (S(x_{n_i},y_{n_i}),S(y_{n_i},x_{n_i}))-(S(x^*,y^*),S(y^*,x^*)) \Vert.
\end{align*}
Proceeding as in the proof of Theorem \ref{theorem2} we have
\begin{align*}
\Vert (x_{n_i},y_{n_i})-(T_{n_i}(x_{n_i},y_{n_i}),T_{n_i}(y_{n_i},x_{n_i})) \Vert&=dist(A,B),\\
\Vert (T_{n_i}(x_{n_i},y_{n_i}),T_{n_i}(y_{n_i},x_{n_i}))-(S(x_{n_i},y_{n_i}),S(y_{n_i},x_{n_i}))\Vert &\leq \frac{1}{n_i}  \Vert (x_0,y_0)-(x_{n_i},y_{n_i}) \Vert,\\
\Vert (S(x_{n_i},y_{n_i}),S(y_{n_i},x_{n_i}))-(S(x^*,y^*),S(y^*,x^*))\Vert& \leq \Vert (x_{n_i},y_{n_i})-(x^*,y^*) \Vert.
\end{align*}
This gives
\begin{align*}
dist(A,B) \leq \Vert (x^*,y^*)-(S(x^*,y^*),S(y^*,x^*)) \Vert & \leq \Vert (x^*,y^*)-(x_{n_i},y_{n_i}) \Vert+ dist(A,B)\\
& \quad +\frac{1}{n_i} \Vert (x_0,y_0)-(x_{n_i},y_{n_i}) \Vert+\Vert (x_{n_i},y_{n_i})-(x^*,y^*) \Vert.
\end{align*}
Letting $i \rightarrow \infty$ we get $\Vert (x^*,y^*)-(S(x^*,y^*),S(y^*,x^*)) \Vert=dist(A,B)$. This implies that $(x^*,y^*)$ is a coupled  best proximity point of $S$. Let $\epsilon>0$ be given and $(u^*,v^*) \in A \times B$ satisfy
\begin{align*}
\Vert u^*-S(u^*,v^*) \Vert &\leq \epsilon+dist(A,B),\\
\Vert v^*-S(v^*,u^*) \Vert & \leq \epsilon+dist(A,B).
\end{align*}
Consider
\begin{align*}
\Vert x^*-u^* \Vert & \leq \Vert x^*-T(u^*,v^*) \Vert + \Vert T(u^*,v^*)-u^* \Vert\\
& \leq \Vert u^*-T(u^*,v^*) \Vert + \Vert T(u^*,v^*) -u^* \Vert\\
& \leq 2 \epsilon+2dist(A,B).
\end{align*}
Similarly, 
$$\Vert y^*-v^* \Vert \leq 2 \epsilon+2dist(A,B)$$
which implies that the coupled  best proximity point problem is Ulam-Hyers stable. 
\end{proof}
\begin{remark}
\emph{It is also intriguing to investigate the existence of a multidimensional best proximity point and  its Ulam-Hyers stability in the case of finitely many nonempty, closed and convex subsets of a uniformly convex Banach space.}
\end{remark}
\section*{Acknowledgements}
 The $^*$corresponding author is supported by UGC Non-NET fellowship (Ref.No. Sch/139/Non-NET/  Math./Ph.D./2017-18/1028).
 
\textbf{Anuradha Gupta}\\
 Department of Mathematics, Delhi College of Arts and Commerce,\\
  University of Delhi, Netaji Nagar, \\
  New Delhi-110023, India.\\
  \vspace{0.2cm}
 email: dishna2@yahoo.in\\
   \textbf{Manu Rohilla}\\
  Department of Mathematics, University of Delhi, \\
  New Delhi-110007, India.\\
  email: manurohilla25994@gmail.com
\end{document}